\documentclass[a4paper,12pt]{amsart}
\usepackage{amsthm,vmargin}
\usepackage{amsmath,xypic}
\begin{document}
\newcommand{\Q}{{\mathbb Q}}
\newcommand{\C}{{\mathbb C}}
\newcommand{\R}{{\mathbb R}}
\newcommand{\Z}{{\mathbb Z}}
\newcommand{\F}{{\mathbb F}}
\renewcommand{\wp}{{\mathfrak p}}
\renewcommand{\P}{{\mathbb P}}
\renewcommand{\O}{{\mathcal O}}
\newcommand{\Pic}{{\rm Pic\,}}
\newcommand{\Ext}{{\rm Ext}\,}
\newcommand{\rank}{{\rm rk}\,}
\newcommand{\sbull}{{\scriptstyle{\bullet}}}
\newcommand{\bX}{X_{\overline{k}}}
\newcommand{\ch}{\operatorname{CH}}
\newcommand{\tors}{\text{tors}}
\newcommand{\cris}{\text{cris}}
\newcommand{\alg}{\text{alg}}
\newcommand{\tX}{{\tilde{X}}}
\newcommand{\tL}{{\tilde{L}}}
\newcommand{\Hom}{{\rm Hom}}
\newcommand{\spec}{{\rm Spec}}
\let\isom=\simeq
\let\rk=\rank
\let\tensor=\otimes
\newcommand{\X}{\mathfrak{X}}
\newcommand{\mydot}{{\small{\bullet}}}
\let\hom=\Hom

\newtheorem{theorem}[equation]{Theorem}      
\newtheorem{lemma}[equation]{Lemma}          %
\newtheorem{corollary}[equation]{Corollary}  
\newtheorem{proposition}[equation]{Proposition}
\newtheorem{scholium}[equation]{Scholium}

\theoremstyle{definition}
\newtheorem{conj}[equation]{Conjecture}
\newtheorem*{example}{Example}
\newtheorem{question}[equation]{Question}

\theoremstyle{definition}
\newtheorem{remark}[equation]{Remark}

\numberwithin{equation}{subsection}

\renewcommand{\t}[1]{\tilde{#1}}
\newcommand{\gpb}[1]{(\t{#1},F_i(\t{#1}))}
\title{Vectors bundles with theta divisors I
\\ Bundles on Castelnuovo curves}
\author{Kirti Joshi and V.~B.~Mehta}
\address{Math. department, University of Arizona, 617 N Santa Rita, Tucson
85721-0089, USA.} \email{kirti@math.arizona.edu}
\address{School of Mathematics, Tata Institute of Fundamental Research,
Mumbai,  India.} \email{vikram@math.tifr.res.in}
\date{Version: Aug 2, 2007}


\begin{abstract}
In this paper we show that semistable vector bundles
on a Castelnuovo curve of genus $g\geq 2$ have theta
divisors. As a corollary, we deduce that semistable
vector bundles on a smooth, general curve of genus
$g\geq 2$ which extend to semistable vector bundles on any
Castelnuovo degeneration of the general curve admit a
theta divisor.
\end{abstract}
\maketitle

\section{Introduction}
\subsection{The classical case}
Let $C$ be a smooth, projective curve of genus $g\geq
2$. We say, following \cite{raynaud82}, that a vector
bundle $V$ on $C$ with $\mu(V)=\deg(V)/\rk(V)=0$, has
a theta divisor if $\hom(L_{gen},V)=0$ for a general
line bundle $L_{gen}$ of degree $1-g$. When such a
line bundle exists, the set of line bundles for which
$\hom(L,V)\neq 0$ forms a divisor in $\Pic^{1-g}(C)$,
which is algebraically equivalent to $\rk(V)\theta$
where $\theta$ is the classical theta divisor of
$\Pic^{1-g}(C)$. It is easy to see that if $V$ has a
theta divisor then $V$ is semistable. It was shown in
\cite{raynaud82} that there are stable vector bundles
on a smooth, projective curve of genus $g\geq 2$
which do not have a theta divisor. Raynaud also
showed that if $\rk(V)=2$ then every semistable
vector bundle has a theta divisor and if $\rk(V)=3$
then the assertion continues to hold if $X$ is
generic. It is well-known that locus of semi-stable vector bundles with theta divisors is open in the moduli of semi-stable vector bundles (see \cite{raynaud82}). 
But at the moment we do not have any concrete description of this locus. 
In the present note we provide an open set of the moduli of
semi-stable vector bundles (on a generic curve) which have theta divisors. In this note we prove the following theorems.

\begin{theorem}\label{main1}
Suppose $\X\to\spec(k[[t]])$ is a flat, proper family
of curves with smooth generic fibre $\X_\eta$ and the
special fibre $\X_0=X$ is a Castelnuovo curve of
arithmetic genus $g\geq 2$.  Suppose $V$ is a
semistable vector bundle on $X_\eta$ of degree zero
which extends as a semistable vector bundle on $\X_0$.
Then $V$ has a theta divisor.
\end{theorem}

Recall that a projective, irreducible nodal curve $X$ of
arithmetic genus $g$ and with $g$ nodes is called a Castelnuovo
curve of genus $g$. By the results of \cite{deligne69} we know
that a smooth, general curve of genus $g$ admits a Castelnuovo
degeneration. Thus the above theorem has the following corollary.

\begin{corollary}\label{DM}
Let $X/k$ be a general smooth, projective curve of
genus $g\geq 2$. Let $V$ be a semistable vector bundle on $X$
which extends to a semistable vector bundle on some
Castelnuovo degeneration of $X$. Then $V$ admits a
theta divisor.
\end{corollary}

\subsection{Theta divisors for semistable bundles on Castelnuovo curves}
Theorem~\ref{main1} is proved by studying vector bundles on a
Castelnuovo curve $X$. By the semi-continuity theorem it is clear
that to prove the Theorem~\ref{main1} it is sufficient to prove
the following theorem.

\begin{theorem}\label{main2}
Let $X$ be a Castelnuovo curve of genus $g\geq 2$.
Then every semistable vector bundle on $X$ with
$\deg(V)=0$ has a theta divisor.
\end{theorem}

\subsection{Theta divisors for generalized parabolic bundles}
Vector bundles and more generally torsion-free
sheaves on a Castelnuovo curve $X$ can be described
as bundles on the normalization $$\pi:\t X=\P^1\to
X$$ with additional structures. There are several
such descriptions available (see
\cite{raynaud82,seshadri82,bhosle92,bhosle96}). The
description we will use is the one developed in
\cite{bhosle92,bhosle96}. In \cite{bhosle92}
torsion-free sheaves on a Castelnuovo curve are
described in terms of \emph{generalized parabolic
bundles}. The advantage of this description is that
the correspondence between bundles on the
normalization  $\t X$ (with additional structures)
and bundles on $X$ preserves degrees of the
underlying bundles. The notion of stability,
semistability of bundles extends easily to
generalized parabolic bundles. Vector bundles on a
Castelnuovo curve $X$ arise from special type of
generalized parabolic bundles which we will call
\emph{generalized parabolic bundles of type $B$}. Any
generalized parabolic bundle of type $B$ gives rise
to a unique (up to isomorphism) vector bundle on $X$
and conversely; moreover semi-stability (resp.
stability) is preserved under this correspondence.
Theorem~\ref{main2} follows from the following

\begin{theorem}\label{main3}
Let $a_i\neq b_i$ be $g\geq 2$ pairs of points on
$\P^1$. Let $\gpb V$ be a generalized parabloic
bundle of type $B$ with generalized parabolic
structures at $\{a_i,b_i\}_{1\leq i\leq g}$. Assume
$\gpb V$ is semistable with $\mu(V)=0$. Then  there
exists a line bundle $\t L$ of degree $1-g$ with a
generalized parabolic structure $\gpb{L}$ of type $B$
such that
$$\hom(\gpb{L},\gpb V)=0.$$
In other words, $\gpb{V}$ has a theta-divisor.
\end{theorem}

\subsection{} We note that in \cite{raynaud82} it was shown that there exists stable torsion-free sheaves on a Castelnuovo curve of genus $g\geq 2$ which do not have a theta divisor. The moduli space of semistable torsion-free sheaves on a Castelnuovo curve exists (see \cite{seshadri82}) and properly contains the locus of locally-free semistable sheaves as an open subset.

\subsection{} In this paper we have treated the case of generalized parabolic bundles on $\t X=\P^1$, in a forthcoming paper \cite{joshi06b} we will treat the case when $\t X$ is an aribitrary smooth, projective curve (not neccessarily of genus $0$). The details are a little more complicated and it was felt by us that, for readability, it would be best to separate them.

\subsection{Acknowledgements} This paper was written while the first author was visiting the Tata Institute  and the author thanks the institute for hospitality. We are also grateful to Usha Bhosle for conversations and correspondence about generalized parabolic bundles.

\section{Preliminaries}
\subsection{Castelnuouvo curves} Let $k$ be an algebraically closed field.
Let $\tX=\P^1$ be the projective line over $k$. Let $a_i\neq b_i$
be $1\leq i\leq g$ pairs of $g$ pairs of points $\tX$. Identifying
$a_i$ with $b_i$ gives us a nodal curve $X$ which is an integral
nodal curve of arithmetic genus $g$ and has exactly $g$ nodes. The
normalization of $X$ is $\tX=\P^1$ and the normalization map is
$\pi:\tX=\P^1\to X$ which identifies the points $a_i$ with $b_i$
for $1\leq i\leq g$.

Conversely, if $X$ is an integral nodal curve of arithmetic genus
$g$ and with $g$ nodes, then it arises in the above manner for
some $g$ pairs of points $a_i\neq b_i$.

We will call, an integral, nodal curve with $g$ nodes, and of
arithmetic genus $g$, a \emph{Castelnuovo curve of genus $g$}.

\subsection{} In what follows,
we will need to work with sheaves on $\t X$
as well as on $X$. So we will adopt the following convention: sheaves on $\t X$ will be denoted with a $\tilde{}$ on top, while sheaves on  $X$ will be denoted without the $\tilde{}$ on top. Thus $\t V,\t L$ etc will be sheaves on $\t X$, and $V, L$ will be sheaves on $X$.

\subsection{Stalks and fibers}\label{stalksandfibres} Let $\t V, \t W$ be vector bundles on $\t X$, let $\t f:\t W\to \t V$ be morphism of $\O_{\t X}$-modules. Let $x\in \t X$ be a point. We will write $V_{\t X,x}$ for the stalk of $\t V$ at $x$ and write $\t V_x=V\tensor_{\O_{\t X,x}} k$ for the fibre at $x$. Recall that a morphism $\t f:\t W\to \t V$ induces a morphism on the stalks at $x$ and also a morphism $W_x \to V_x$ on the fibres (this morphism may be identically zero!).

\subsection{Generalized parabolic bundles}
We shall use the formalism of generalized parabolic
bundles of \cite{bhosle92}.
Since this is not so well-known we recall facts proved in \cite{bhosle92,bhosle96}. A \emph{generalized parabolic bundle} (a GPB for short) on $\tX$ is a collection $(\t{V},\{F_i(\t V)\}_{1\leq i\leq g}\})$ where $\t V$ is a vector bundle on $\t X$ and $F_i(\t V)\subset \t V_{a_i}\oplus \t V_{b_i}$ is a subspace.  We will simply write $(\t V, F_i(\t V))$ instead of $(\t{V},\{F_i(\t V)\}_{1\leq i\leq g}\})$ for simplicity.

\subsection{Morphisms of generalized parabolic bundles} A \emph{morphism of generalized parabolic bundles} $f:\gpb W \to \gpb V$ is a morphism $f: W\to V$ of $\O_{\t X}$-modules and a map of vector spaces $F_i(\t W) \to F_i(\t V)$ for $1\leq i\leq g$ such that the diagram
$$
\xymatrix{
  F_i(\t W) \ar[d]_{} \ar[r]^{} & \t W_{a_i}\oplus \t W_{b_i} \ar[d]^{} \\
  F_i(\t V) \ar[r]^{} & \t V_{a_i}\oplus \t V_{b_i}   }
$$
commutes. Here the vertical arrow on the right is the direct sum
of the natural arrow (see \ref{stalksandfibres}) $\t W_{a_i}\to \t
V_{a_i}$ and the corresponding arrow for $b_i$.

\subsection{Induced GPB structures}\label{induced-structures} Suppose $\gpb V$ is a GPB and $\t W\subset \t V$ is a subsheaf. Then $W$ carries an \emph{induced GPB structure}: define, for $1\leq i\leq g$, a subspace
$F_i^{ind}(\t W)\subset\t W_{a_i}\oplus \t W_{b_i}$ as the
equalizer of the two arrows to $\t V_{a_i}\oplus\t V_{b_i}$:
\begin{equation}
\xymatrix{
   & \t W_{a_i}\oplus \t W_{b_i} \ar[d]^{} \\
  F_i(\t V) \ar[r]^{} & \t V_{a_i}\oplus \t V_{b_i}   }
\end{equation}
Explicitly $F_i^{ind}(\t W)$ consist of elements of $\t
W_{a_i}\oplus \t W_{b_i}$ whose image lies in $F_i(\t V)$. The
diagram
$$
\xymatrix{
  F_i^{ind}(\t W) \ar[d]_{} \ar[r]^{} & \t W_{a_i}\oplus \t W_{b_i} \ar[d]^{} \\
  F_i(\t V) \ar[r]^{} & \t V_{a_i}\oplus \t V_{b_i}   }
$$ commutes by properties of equalizers.
This equips the bundle $\t W$ with GPB structure, which we call
the \emph{induced GPB structure} on $\t W$. Obviously we have a
morphism $\gpb W \to \gpb V$.

We caution the reader that our definition of GPB structure is
slightly different with that of \cite{bhosle92}, where $\t
W\subset \t V$ is a subbundle, in which case the maps on the
fibres is an inclusion, so we can talk about intersections of
subspaces of fibres of $\t V$. It is easy to see that our
definition agrees with loc. cit. in this case, and works in all
cases.

\subsection{Generalized parabolic bundles of types $B$ and  $T$} A generalized parabolic bundle $\gpb V$ is said to be of type $B$ if the composite maps
\begin{equation}
\xymatrix{
& & \t V_{a_i}\\
F_i(\t V)\ar[r]\ar[urr]^\simeq\ar[drr]_\simeq & \t V_{a_i}\oplus \t V_{b_i} \qquad\ar[ur]\ar[dr] & \\
& & \t V_{b_i} }
\end{equation}
are both isomorphisms. A GPB which is not of type $B$ is called a
GPB of type $T$. The terminology has its origins in the nature of
the torsion-free sheaves which these bundles give rise to. A GPB
of type $B$ gives rise to a locally-free sheaf on $X$, while a GPB
of type $T$ gives rise to a torsion-free but non-locally-free
sheaf on $X$. Thus type $B$ is short for ''bundle type'', while
type $T$ is everything torsion-free but not bundle type. We
caution the reader that this terminology was not introduced in
\cite{bhosle92}.

\subsection{} The following Lemma is immediate from \cite{bhosle92} (also see \cite{raynaud82,seshadri82}).
\begin{lemma}
Let $\gpb V$ be a generalized parabolic bundle on $\t X$. Then the
following are equivalent:
\begin{enumerate}
\item the GPB $\gpb V$ is of type $B$,
\item for every $i$, $1\leq i\leq g$, the subspace $F_i(\t V)\subset \t V_{a_i}\oplus \t V_{b_i}$ is the graph of an isomorphism $\t V_{a_i}\to \t V_{b_i}$.
\end{enumerate}
\end{lemma}

\subsection{Degrees and slopes} Let $\t V$ be a vector bundle on $\t X$. Then its usual degree on $\t X$ will be denoted by $\deg(\t V)$, while its slope $\mu(\t V)=\deg(\t V)/\rk(\t V)$.

 Now let $\gpb V$ be a GPB on $\tX$. Then its \emph{generalized parabolic degree} is given by
\begin{equation}
\deg(\gpb V)=\deg(\t V)+\sum_{i=1}^g \dim F_i(\t V).
\end{equation}
The \emph{generalzed parabolic  slope} of $\gpb V$ is a defined by
\begin{equation}
\mu(\gpb V)=\frac{\deg(\gpb V)}{\rk(\t V)}=\frac{\deg(\t
V)+\sum_{i=1}^g \dim F_i(\t V)}{\rk(\t V)}.
\end{equation}

\subsection{} Note, in particular, that if $\gpb V$ is a generalized parabolic bundle of type $B$ and rank $r$ then we have the following formula:
\begin{equation}
\deg(\gpb V)=\deg(\t V)+gr.
\end{equation}
Hence we have
\begin{equation}
\mu(\gpb V)=\mu(\t V)+g.
\end{equation}
In particular, if $\gpb L$ is a GPB of type $B$ with $\rk (\t
L)=1, \deg(\t L)=1-g$, then we have
\begin{equation}
\mu(\gpb L)=\deg(\gpb L)=1.
\end{equation}

\subsection{Stability and semi-stability} We say that a generalized Parabolic bundle $\gpb V$ is semistable (resp. stable), if for every $\gpb W\subset \gpb V$, we have
$\mu(\gpb W)\leq \mu(\gpb V)$ (resp. $\mu(\gpb W)< \mu(\gpb V)$).

\subsection{Torsion-free sheaves on Castelnuovo curves}
\label{gpbdescent} Let $\gpb V$ be a GPB on $\t X$. Then the subsheaf $V$ defined by the exact sequence
\begin{equation}
0\to V\to \pi_*(\t V) \to \oplus_{i=1}^g\frac{\t V_{a_i}\oplus \t
V_{b_i}}{F_i(\t V)}\to 0
\end{equation}
is a torsion free sheaf on $X$ with $\rk(V)=\rank(\t V)$.
Conversely given any torsion-free sheaf on $X$, there is a GPB on
$\t X$ which gives rise to it in this manner. The following
theorem is from \cite{bhosle92}.

\begin{theorem} Under the correspondence $\gpb{V}\mapsto V$, a generalized parabolic bundle of type $B$ gives rise to a vector bundle $V$ and conversely every vector bundle on $X$ arises from a GPB of type $B$ on $\t X$. This GPB of type $B$ is unique upto  isomorphism.
\end{theorem}

\section{Theta divisors for generalized parabolic bundles}
\subsection{} We extend the notion of theta divisors of vector
bundles to theta divisors of generalized parabolic bundles.

\subsection{} We will say that a generalized parabolic bundle
$\gpb V$ has a \emph{generalized parabolic theta divisor} if there
exists a GPB of type $B$ with $\rk(\t L)=1$ such that
$$\hom(\gpb L,\gpb V)=0.$$

\subsection{} If no such line bundle exists for a $\gpb V$, then every GPB of type $B$ and rank one admits a non-zero map to $\gpb V$. This is similar to the situation for vector bundles.

\subsection{} We now explain why the nomenclature ``theta divisor'' is
appropriate. To do this we need to explicate the condition in the
definition. We may do this in $\Pic(X)$, rather than in the scheme
of GPB line bundles of type $B$ on $\t X$ (which is a
$\P^1$-bundle on the compactification of $\Pic(X)$).

Suppose $\gpb L$ is a GPB of type $B$ on $\t X$ with $\rk(\t L)=1$
(so it gives rise to a line bundle $L$ on $X$. Moreover $F_i(\t
L)\subset L_{a_i}\oplus L_{b_i}$ is the graph of some isomorphism
$L_{a_i}\to L_{b_i}$. Such an isomorphism is given by a non-zero
scalar, say $\lambda_i$.

Now suppose $\gpb V$ is a GPB of type $B$ and suppose that we have
a morphism of GPB's $\gpb L \to \gpb V$. Then we have a
commutative diagram
\begin{equation}
\xymatrix{
F_i(\t L) \ar[r]\ar[d] & \t L_{a_i}\oplus \t L_{b_i}\ar[d]\\
F_i(\t V) \ar[r] & \t V_{a_i}\oplus \t V_{b_i}\\
}
\end{equation}
Let $F_i(\t V)$ be the graph of an isomorphism $A_i:\t V_{a_i}\to
\t V_{b_i}$, for $1\leq i\leq g$. Then the commutativity of the
diagram forces the condition
\begin{equation}
\prod_{i=1}^g \det(A_i-\lambda_i I)=0
\end{equation}
This is a divisorial condition in the space of GPB bundles of type
$B$ and rank one and the above equation is the equation of the
theta divisor (when it is not all of $\Pic(X)$).

\section{A zero lemma}
\subsection{} In this section we prove a couple of preparatory lemmas for the next section.
\begin{proposition}\label{non-vanishing}
Let $\gpb V$ be a semistable GPB of type $B$ on $\t
X$, assume $\deg(\t V)=0$. Suppose $\gpb L$ is a GPB
of type $B$ with $\rk(\t L)=1$ and $\deg(\t L)=1-g$.
Suppose $\gpb L\to \gpb V$ is any non-zero map.
Assume that the underlying morphism of $\O_{\t
X}$-modules  $\t L \to \t V$ vanishes at $d$ of the
$2g$ point $\{a_i,b_i\}_{1\leq i\leq g}$. Then $d\leq
2g-1$.
\end{proposition}
\begin{proof}
If the map $\t L\to \t V$ vanishes at $d$ points,
then this map factors as $\t L \to \t
M=\O_{\P^1}(1-g+d)\to \t V$. Now give $\t
M=\O_{\P^1}(1-g+d)$ the induced GPB structure (see
\ref{induced-structures} $F_i^{ind}(\t M)\subset \t
M_{a_i}\oplus \t M_{b_i}$. Now we have
\begin{equation} \mu(\gpb M)=\deg(\t M)+\sum_i \dim
F_i(\t M).
\end{equation}
 So in any case $\mu(\gpb M)\geq 1-g+d$. On the other hand $\gpb V$ is a
 semistable GPB of type $B$. So its generalized parabolic slope is
\begin{equation}
\mu(\gpb V)=\deg(\t V)+g=g.
\end{equation}
Thus by semi-stability we see that $$1-g+d\leq
\mu(\gpb{M})\leq\mu(\gpb V)=g,$$ that is, $1-g+d\leq
g$. Hence $d\leq 2g-1$. This proves the proposition.
\end{proof}

\section{Theta divisors for GPB bundles of type $B$}
\subsection{} Now we are ready to prove Theorem~\ref{main3}. Let us recall the statement.
\begin{theorem}\label{main4}
Let $\gpb V$ be a semistable GPB of type $B$ on $\t
X$. Suppose that $\deg(\t V)=0$ and $g\geq 2$. Then
there exists a GPB $\gpb L$, with $\deg(\t L)=1-g$
and $\rk(\t L)=1$ such that
\begin{equation}
\hom(\gpb L,\gpb V)=0.
\end{equation}
 In other words $\gpb V$ has a theta divisor.
\end{theorem}

\begin{proof}[Proof of Theorem~\ref{main4}]
We will construct such a $\gpb L$ on $\t X$ with the asserted
property. Since $\deg(\t L)=1-g$, we take $\t L=\O_{\P^1}(1-g)$.
So we need to construct a GPB structure on $\t L$ such that there
are no maps $\gpb L\to \gpb V$. Observe that by the Riemann-Roch
theorem, $$\chi(\t V(g-1))=\rk(\t V)(g-1)+\rk(\t V)(1-0)\neq 0,$$
so that $\hom(\t L, \t V)\neq 0$.

  Let $A_i:\t V_{a_i}\to \t V_{b_i}$ be the isomorphism whose graph is $F_i(\t V)\subset \t V_{a_i}\oplus \t V_{b_i}$. Choose, for $1\leq i\leq g$, non-zero scalars $\lambda_i\in k^*$ such that $\lambda_i$ is not an eigenvalue of $A_i$. Let $F_i(\t L)$ be the graph of the isomorphism $\lambda_i: \t L_{a_i}\to \t L_{b_i}$. Then we claim that
$$\hom(\gpb L, \gpb V)=0.$$

Suppose this is not the case. Then there is a non-zero map
$$f:\gpb L\to \gpb V.$$ Consider the underlying map of $\O_{\t
X}$-modules $f:\t L \to \t V$. By
Proposition~\ref{non-vanishing} we know that the map
$f$ can vanishes at no more than $2g-1$ of the $2g$
points $a_i,b_i$ for $1\leq i\leq g$. Assume if
possible that the map $f$ vanishes at all except one
points, say, $a_1$, and by assumption we see that the
map $f$ must be vanishing at $b_1$.

As $\gpb L\to \gpb V$ is a morphism of GPB
structures, the diagram

$$
\xymatrix{
  F_i(\t L) \ar[d]_{} \ar[r]^{} & \t L_{a_1}\oplus \t L_{b_1} \ar[d]^{} \\
  F_i(\t V) \ar[r]^{} & \t V_{a_1}\oplus \t V_{b_1}   }
$$
commutes. On the other hand, the horizontal arrows
are inclusions and the component $\t L_{a_1}\to \t
V_{a_1}$ is non-zero. So the first vertical arrow
must also be non-zero. Now the composite map $F_1(\t
L)\to \t V_{a_1}\oplus \t V_{b_1}$ is also non-zero.
This is because $F_1(\t L)$ is the graph of the
isomorphism $\lambda_1:\t L_{a_1}\to \t L_{b_1}$ and
so the image of $F_1(\t L)$ in the direct sum is not
in the kernel, $L_{b_1}$, of $\t L_{a_1}\oplus \t
L_{b_1}\to \t V_{a_1}\oplus \t V_{b_1}$. So we
deduce, by commutativity of the diagram, that $F_1(\t
L)\to F_1(\t V)$ is also non-zero.

Let $\t\ell$ be a basis of $F_1(\t L)$ (note that
this is a one dimensional space). As the diagram
commutes, the image of $\t\ell$ is, on one hand,
equal to
$\t\ell\mapsto(\t\ell,\lambda_1\t\ell)\mapsto(\t
v,0)$. While on the other hand, it is given by
$\t\ell\mapsto \t w\mapsto(\t u,A_1\t u)$ where
$F_1(\t V)$ is the graph of $A_1:\t V_{a_1}\to \t
V_{b_1}$. So we must have $(\t v,0)=(\t u,A_1\t u)$.
Hence we must have $\t v=\t u$ and $A_1\t u=0$. But
$A_1:\t V_{a_1}\to \t V_{b_1}$ is an isomorphism. So
$A_1\t u=0$ gives $\t u=0$. Hence the image of
$\t\ell$ under the composite $F_1(\t L)\to F_1(\t
V)\to \t V_{a_1}\oplus\t V_{b_1}$ is zero. But this
is a contradiction. Hence our assumption that $f:\t
L\to \t V$ vanishes at exactly one point must be
wrong.

So we see that $f$ must be non-vanishing  at  two or more
points  and further, the same argument as above shows
that the set of points at which it is non-vanishing
must contain a pair of the form $a_i,b_i$ for some
$i$.

Thus we can assume that $\t L \to \t V$ induces a
non-zero map at the stalks at a pair of points, say
$a_1,b_1$ (after a renumbering these points if
required), with a torsion-free cokernel. Thus $\t L
\to \t V$ is a bundle map at $a_1,b_1$. So we have an
inclusion of fibres $\t L_{a_1}\subset \t V_{a_1}$
and similarly at the point $b_1$. So we can identify
the fibres of $\t L$ at these points as subspaces of
the fibres of $\t V$ at the corresponding points. Now
the fact that we have a map of generalized parabolic
bundles means that we have a commutative diagram
$$
\xymatrix{
  F_i(\t L) \ar[d]_{} \ar[r]^{} & \t L_{a_i}\oplus \t L_{b_i} \ar[d]^{} \\
  F_i(\t V) \ar[r]^{} & \t V_{a_i}\oplus \t V_{b_i}   }
$$
commutes and $\t L_{a_i}\subset \t V_{a_i}$ etc. But this is not
possible: as $F_i(\t L)$ is the graph of $\lambda_i$ and $F_i(\t
V)$ is the graph of $A_i$, so the commutativity forces $\lambda_i$
to be an eigenvalue of $A_i$ and  as $\lambda_i$ is not an
eigenvalue of $A_i$, by choice, so we have a contradiction. This
proves the assertion.
\end{proof}
\begin{corollary}
Let $V$ be a semistable bundle of degree zero on the
Castelnuovo curve $X$ of genus $g\geq 2$. Then $V$
has a theta divisor, that is there exists a line
bundle on $X$, of degree $1-g$, such that
$\hom(L,V)=0$.
\end{corollary}
\begin{proof} Let $\gpb V$ be the GPB bundle of type $B$ on $\t X$
associated to $V$ by \ref{gpbdescent}. Then by the theorem, there
exists $\gpb L$ such that $$\hom(\gpb L,\gpb V)=0.$$. We claim
that there is a natural map
$$\Hom(\gpb L,\gpb V)\to \hom(L,V).$$
To prove this we observe that we have a commutative
diagram with exact rows
 $$
\xymatrix{
  0 \ar[r] & L \ar[r]\ar[d] & \pi_*(\t L) \ar[r]\ar[d] & \oplus_i \frac{\t L_{a_i}\oplus \t L_{b_i}}{F_i(\t L)} \ar[r]\ar[d] & 0\\
 0 \ar[r] & V \ar[r] & \pi_*(\t V) \ar[r] & \oplus_i \frac{\t V_{a_i}\oplus \t V_{b_i}}{F_i(\t L)} \ar[r] & 0
}
$$
In the diagram, the right most vertical arrow is the direct sum of
the arrows, one for each $i$, for $1\leq i\leq g$. The arrow for
each $i$ is given using the fact that the map $F_i(\t L)\to F_i(\t
V)$. The middle arrow is the push-forward of the arrow $\t L\to \t
V$ underlying the map $\gpb L\to \gpb V$. Thus we have constructed
the arrow $L\to V$ given a map of $\gpb L\to \gpb V$.

We claim that $\hom(\gpb L,\gpb V)\to \hom(L,V)$ is onto. So
suppose that we have a map $f:L\to V$ we have to construct a
corresponding map of $\gpb L\to \gpb V$ which gives rise to it in
the above fashion. Pulling back the map $f$ by $\pi$, we get a
$\pi^*(f):\pi^*(L)\to \pi^*(V)$ and noting that $\pi^*(L)=\t L$
and $\pi^*(V)=\t V$, we deduce that we have a map $$\t f:\t L\to
\t V,$$ Now we have to verify that we also have a map of parabolic
data at the points $a_i,b_i$ for $1\leq i \leq g$. But as the map
$\t f$ which we just constructed is pulled back from $X$, it comes
equipped with gluing data at the $a_i,b_i$ for $1\leq i\leq g$,
which gives maps of the parabolic data at the points over the
nodes of $X$. Further as $L,V$, are locally free sheaves on $X$,
by \cite{bhosle92}, we know that $L,V$ arise from unique
generalized parabolic bundles on $\t X$ of type $B$. Thus, we see
that we have constructed a non-zero map of generalized parabolic
bundles $\gpb L \to \gpb V$, and this is a contradiction.

Thus we have completed the proof of the corollary.
\end{proof}
\subsection{} Now we are ready to complete the proof of Theorem~\ref{main1}. We recall what we have to prove:

\begin{theorem}\label{main5} Suppose $\X\to\spec(k[[t]])$
is flat, proper family of curves with smooth generic
fibre $\X_\eta$ and the special fibre $\X_0=X$ is a
Castelnuovo curve of arithmetic genus $g\geq 2$.
Suppose $V$ is a semistable vector bundle on $X_\eta$ of
degree zero which extends as a semistable vector
bundle on $\X_0$. Then $V$ has a theta divisor.
\end{theorem}

\begin{proof}[Proof of Theorem~\ref{main5}]
Now we are ready to prove Theorem~\ref{main5}. Suppose
$\X\to\spec(k[[t]])$ is a flat family of curves with smooth,
projective generic fibre of genus $g$, denoted $\X_\eta$ and the
special fibre $\X_0=X$ is a Castelnuovo curve of genus $g$.
Suppose $V_\eta$ is a vector bundle of degree zero on $\X_\eta$
which extends to a vector bundle $V$ on $X$. By
Theorem~\ref{main2}, there exists a line bundle $L$ on $X$, of
degree $1-g$ on $X$ such that $\hom(L,V)=0$. As
$\Ext^2(\O_X,\O_X)=0$, $L$ extends to a line bundle $L_\eta$ on
$\X_\eta$ and by the semi-continuity theorem we see that
$\hom(L_\eta,V_\eta)=0$ as $\hom(L,V)=0$. This proves the theorem.
\end{proof}
\subsection{} Now Corollary~\ref{DM} can also be established.
\begin{corollary}\label{DM2}
Let $X/k$ be a general smooth, projective curve of
genus $g\geq 2$. Let $V$ be a semistable vector bundle on $X$
which extends to a semistable vector bundle on some
Castelnuovo degeneration of $X$. Then $V$ admits a
theta divisor.
\end{corollary}

\begin{proof}[Proof of Corollary~\ref{DM2}]
By \cite{deligne69} we know that a general curve of genus $g\geq
2$ admits a Castelnuovo degeneration and indeed many such
degenerations. Thus if $V$ is a vector bundle which extends to
some Castelnuovo degeneration, then Theorem~\ref{main5} applies
and we are done.
\end{proof}
\providecommand{\bysame}{\leavevmode\hbox to3em{\hrulefill}\thinspace}
\providecommand{\MR}{\relax\ifhmode\unskip\space\fi MR }
\providecommand{\MRhref}[2]{%
  \href{http://www.ams.org/mathscinet-getitem?mr=#1}{#2}
}
\providecommand{\href}[2]{#2}

\end{document}